\documentclass[12pt]{article}

\usepackage{graphicx,color}  
\usepackage{amsmath,amssymb,amsthm}  
\usepackage{amsfonts,mathtools}
\usepackage{bm}  
\usepackage[english]{babel}
\usepackage[comma,authoryear]{natbib}
\usepackage{todonotes}

\theoremstyle{plain}

\newtheorem{theorem}{Theorem}[section]
\newtheorem{proposition}[theorem]{Proposition}

\newtheorem{corollary}[theorem]{Corollary}

\theoremstyle{definition}
\newtheorem{example}{Example}[section]
\newtheorem{definition}[theorem]{Definition}
\newtheorem{remark}[theorem]{Remark}
\usepackage{titlesec}
\usepackage[titletoc,toc,title]{appendix}

\DeclareMathOperator{\CNormal}{\mathcal {CN}}

\DeclareMathOperator{\Expectation}{\mathbb E}
\DeclareMathOperator{\Normal}{\mathcal N}

\newcommand{\adjointof}[1]{{#1}^*}
\newcommand{\avalof}[1]{\left\vert#1\right\vert}
\newcommand{\cnormalof}[2]{\CNormal_{#1}\left(#2\right)}
\newcommand{\complexu}{\mathbf i}
\newcommand{\complex}{\mathbb C}
\newcommand{\conjugateof}[1]{\overline{#1}}

\newcommand{\euler}{\mathrm{e}}
\newcommand{\expectof}[1]{\Expectation\left(#1\right)}
\newcommand{\expof}[1]{\exp\left(#1\right)}
\newcommand{\normalof}[3]{\Normal_{#1}\left(#2,#3\right)}
\newcommand{\partiald}[2]{\frac{\partial}{\partial #1} #2}

\newcommand{\reals}{\mathbb{R}}
\newcommand{\set}[1]{\left\{#1\right\}}
\newcommand{\smallo}{\operatorname{o}}

\newcommand{\transposedof}[1]{{#1}^t}

\title{Moments of the Complex Multivariate Normal Distribution}
\author{C. Fassino\thanks{Dipartimento di Matematica, Universit\`a di Genova, Genova, Italy }, G. Pistone\thanks{de Castro Statistics, Collegio Carlo Alberto, Piazza Vincenzo Arbarello 8, 10122 Torino, Italy}, M.P. Rogantin$^{*}$}

\begin{document}
\maketitle

\abstract{We present  a characterization of the null moments of the Complex Multivariate Normal Distribution with non-singular covariance matrix and we give closed-forms expressions for its non-null moments.}

\section{Introduction}\label{sec:intro}
Complex Multivariate Normal Distribution (CMND) is defined by \citet{ito:1952} to be the image under the affine tranformation $z \mapsto \mu + Az$ of a standard CMND, where $z, \mu \in \complex^p$, $A \in \complex^{p\times p}$. In turn, a (univariate) Complex Normal Distribution (CND) is a zero-mean normal distribution of $C \simeq \reals^2$ which is invariant under all rotations $z \mapsto \euler^{\complexu\theta}z$ e.g., such that its real and imaginary components are independent and equally distributed. A standard CMND is a vector of independent identically distributed CND. In the following all CMND are assumed to be centered, $\mu=0$. Applications of CMND have been discussed e.g., by \citet{Goodman1963}.

As in the real case, CMND are characterized by a complex covariance matrix  $\Sigma = \expectof{Z\adjointof Z}$, which is self-adjoint, $\adjointof \Sigma = \Sigma$, and we assume to be positive definite, $\adjointof z \Sigma z > 0$ if $z \ne 0$. We denote the CMND with zero mean and variance $\Sigma$ by $\cnormalof p \Sigma$.

The set of all (complex) monomials $\prod_{j=1}^p z_j^{n_j} \conjugateof z_j^{m_j}$ is a separating set of functions on $\complex^p$ which is stable for product and conjugation. The expected values of the complex monomials under $\cnormalof p \Sigma$ are the complex moments whose computation as a functions of the covariance matrix $\Sigma$ is the object of this paper. Finding a low complexity algorithm for the computation of the value of the moments and giving simple conditions under which the moment is zero is non-trivial and of interest.

Various approaches are possible. The first option is the computation of the moments in the standard case, followed by the multi-linear computations induced by the linear transformation of the variables in the monomials. A second option is to consider the characteristic function of $\cnormalof p \Sigma$ and its relation with complex moments, as it is done by \citet{sultan|tracy:1996}. Here, we follow a third option, namely the generalization to the complex field of Stein equality: if $X \sim \normalof 1 0 1$, then $\expectof {f'(X)} = \expectof{Xf(X)}$. Repeated applications of such equation produce  recurrent relations.

Starting from the recurrent relations on the moments, first, we derive sufficient conditions on the complex monomial implying that  the corresponding complex moment is zero, see Theorem~\ref{th:conditions}. Such conditions apply to a general covariance matrix $\Sigma$ and specialize to the less straightforward cases where some among the elements of $\Sigma$ or  the exponents are zero. Second, we show that, precisely when the sufficient condition for the moment to be zero are not satisfied, it is possible to fully describe the reduction algorithm and eventually devise and prove a closed form for the value of the moment, see Theorem~\ref{th:main}.

The case where $\Sigma$ is the identity matrix, considered in \citet[Sec. 4]{Barvinok07}, is a special case of our results.
\citet{sultan|tracy:1996} consider the general $\Sigma$ and derive, from the characteristic function, recurrent relations for the moments in a tensor form, which appears to be computationally of high complexity if  some specific moments have to be computed.
Conditions for the nullity of the moments are not discussed.

The paper is organised as follows. In Sec.~\ref{sec:MCN} we review basic properties of the MCND. Complex moments are discussed in Sec.~\ref{sec:complexmoments}. We derive a recurrent relation, some sufficient conditions for the moments to be zero, and finally an explicit expression of moments. Part of the proofs is given in Appendix. A short final discussion closes the paper.

\section{Complex Multivariate Normal Distribution}\label{sec:MCN}

Due to the identification $\complex \leftrightarrow \reals^2$, a $p$-variate CMND corresponds, in the real space $\reals^{2p}$, to a special case of a $2p$-variate Normal Distribution. The image in $\complex^p$ of the general Multivariate Normal in $\reals^{2p}$  is also considered in the literature, see e.g. in  \citet{Goodman1963}, where our special case is said to have the property of circular invariance. We discuss first this issue.

We observe that the identification of the number field $\complex$ with the real vector space $\reals^2$ endows the complex number field with the further structure of real algebra (a vector space with multiplication) of dimension 2, with notable consequences on the definition of linear mapping.
Notice that the scalar product on $\reals^2$ becomes $ z_1 \cdot z_2 = \Re(z_1\conjugateof{z_2})$ in $\complex$, which is not the same as the Hermitian scalar product $z_1 \overline z_2$.

Consider the linear mapping $A \colon z \mapsto \alpha z$ with $\alpha = \alpha_1+ \complexu \alpha_2$. In $\reals^2$ we have
\begin{equation*}
  (x,y) \mapsto x+ \complexu y \xmapsto{A} (\alpha_1 + \complexu \alpha_2) (x+{\mathbf i}y) =
  (\alpha_1x-\alpha_2y) + {\mathbf i} (\alpha_2x+\alpha_1 y) \mapsto \\
  \begin{bmatrix}
    \alpha_1 & - \alpha_2 \\ \alpha_2 & \alpha_1
  \end{bmatrix}
  \begin{bmatrix}
    x \\ y
  \end{bmatrix} \ ,
\end{equation*}
where the matrix is  not  a generic one.

However, the linear mapping $ A \colon z \mapsto \alpha z + \beta \conjugateof z, \quad \alpha,\beta \in \complex $
is generic. In fact, in $\reals^2$:
\begin{equation*}
  \begin{bmatrix}
    x \\ y
  \end{bmatrix} \xmapsto{A}
  \begin{bmatrix}
    \alpha_1 & - \alpha_2 \\ \alpha_2 & \alpha_1
  \end{bmatrix}
  \begin{bmatrix}
    x \\ y
  \end{bmatrix} +
  \begin{bmatrix}
    \beta_1 & -\beta_2 \\ \beta_2 & \beta_1
  \end{bmatrix}
  \begin{bmatrix}
    x \\ - y
  \end{bmatrix} =
  \begin{bmatrix}
    \alpha_1+\beta_1 & - \alpha_2 + \beta_2 \\ \alpha_2+\beta_2 & \alpha_1-\beta_1
  \end{bmatrix}
  \begin{bmatrix}
    x \\ y
  \end{bmatrix} \ ,
\end{equation*}
and
\begin{equation*}
  A = \alpha_1
  \begin{bmatrix}
    1 & 0 \\ 0 & 1
  \end{bmatrix} + \alpha_2
  \begin{bmatrix}
    0 & -1 \\ 1 & 0
  \end{bmatrix} + \beta_1
  \begin{bmatrix}
    1 & 0 \\ 0 & -1
  \end{bmatrix} + \beta_2
  \begin{bmatrix}
     0 & 1 \\ 1 & 0
  \end{bmatrix} \ .
\end{equation*}

One easily verifies that the transpose of the linear mapping $A$ with respect to the scalar product is the linear mapping
$ \transposedof A \colon z \mapsto \conjugateof \alpha z + \beta \conjugateof z
$
and the corresponding matrix is the transpose of the matrix $A$.

\subsection{Derivatives and integration by part}
\label{sec:derivatives}

If $\complex$ is see as vector space on $\reals$, then a function $f \colon \complex \to \complex$ is differentiable at $z$ if there exists a linear mapping $df(z)$ such that $f(z+w) - f(z) - df(z)[w] = \smallo(w)$.
Because of the form  of linear mappings, we have $df(z)[w] = \alpha(z)w + \beta(z)\conjugateof w$.

If we assume $f \in C^{1,1}(\reals^2;\complex)$, then an easy computation shows that $\alpha(x,y) = D^-f(x,y)$, $\beta(x,y) = D^+f(x,y)$ with
\begin{equation*}
D^-f(x,y)= \frac12 \left(\partiald x {f(x,y)} - \mathbf{i}  \partiald y {f(x,y)}\right), \quad D^+f(x,y) = \frac12 \left(\partiald x {f(x,y)} +\mathbf{i}  \partiald y {f(x,y)} \right) \ ,
\end{equation*}
see e.g. \cite[Ch. 11]{rudin:1987RCA-III}. Sometimes,  these differential operators written $\partiald {z} {f(z)} $ and $ \partiald {\conjugateof z} {f(z)} $, respectively, because of the special case of complex monomials,
\begin{equation}\label{eq:complexmonomials}
  \partiald z {z^n\conjugateof z^m} = nz^{n-1}\conjugateof z^m, \quad \partiald {\conjugateof z} {z^n\conjugateof z^m} = m z^n \conjugateof z^{m-1} \ .
\end{equation}

If the function $f$ is real valued, $f \colon \complex \to \reals$, then $df(z) \in L(\complex,\reals)$, hence $\conjugateof{\partiald {z} {f(z)}} = \partiald {\conjugateof z} {f(z)}$ necessarly, so that
\begin{equation*}
  df(z)[w] = \partiald {z} {f(z)} w + \conjugateof{\partiald {z} {f(z)} w} = 2 \Re\left(\partiald {z} {f(z)} w \right) = \partiald x {f(x,y)}w_1 + \partiald y {f(x,y)} w_2 \ .
\end{equation*}

For example, consider a real valued function to be discussed below, $\varphi(z) = \frac 1\pi \euler^{-z\conjugateof z}$. We have $\tilde \varphi(x,y) = \frac1\pi\euler^{-(x^2+y^2)}$ and, by the derivative of the composition,
\begin{equation*}
  \partiald z {\varphi(z)} = - \frac12\left(\tilde \varphi(x,y)(2x) - {\mathbf i} \tilde\varphi(x,y)(2y)\right) = - \conjugateof z \varphi(z) \ .
\end{equation*}
Similarly, $\partiald {\conjugateof z} {\varphi(z)} = - z \varphi(z)$, hence the equality $\conjugateof{\partiald {z} {\varphi(z)}} = \partiald {\conjugateof z} {\varphi(z)}$ holds.

We have the following properties.

\begin{proposition}
\label{prop:rudin}
  \begin{enumerate}
\item \label{item:rudin1} The operators $D^-$ and $D^+$ commute and
  \begin{equation*}
    D^-D^+ f = D^+D^- f = \frac12 \Delta f \ .
  \end{equation*}
\item \label{item:rudin2} Both $D^-$ and $D^+$ are derivation operators i.e.,
  \begin{equation*}
    D(fg) = (Df)g + f(Dg), \quad \text{for} \ D = D^-, D^+ \ .
  \end{equation*}
\item  \label{item:rudin3}  The operators $D^-$ and $D^+$ are related by $D^- \overline{g} = \overline{D^+ g} $.
\item \label{item:rudin4} $f$ is holomorphic if, and only if, the Cauchy-Riemann equation hold that is, $D^+f = 0$.
\end{enumerate}
\end{proposition}

\begin{proof}
  Items \ref{item:rudin1}, \ref{item:rudin2} and  \ref{item:rudin3}  are derived by simple computations. See \cite[Ch. 11]{rudin:1987RCA-III} for Item \ref{item:rudin4}.
\end{proof}

We derive an integration by parts equation to be used later. In the following
$\mathcal S(\complex)$ denotes the $\complex$-algebra of $C^\infty$ functions which are bounded, together with all derivatives, by a polynomial and $\mathcal D (\complex)$ the $C^\infty$ functions which are decreasing more than any polynomial. Let us compute the adjoint of both $D^-$ and $D^+$ in the standard Hermitian product of $\complex$.

\begin{proposition}\label{pr:int_part}
Given $f\in\mathcal S(\complex)$ and $g \in \mathcal D(\complex)$,
\begin{eqnarray*}
  \int D^-f(z) \ \conjugateof {g(z)}  \ dz &=&- \int f(z) \ D^- \conjugateof {g(z)}  \ dz \\ \nonumber
  \int D^+f(z) \ \conjugateof {g(z)}   \ dz & = & - \int f(z) \ D^+\conjugateof {g(z)} \ dz \ .
\end{eqnarray*}
\end{proposition}
\begin{proof}
\begin{multline*}
  \int D^-f(z) \conjugateof {g(z)} \ dz =
\int \frac12\left(\partiald x f(x,y) - \complexu \partiald y f(x,y)\right) \conjugateof {g(x,y) } \ dx dy = \\
- \int f(x,y) \frac12 \left( \partiald x - \complexu \partiald y  \right) \conjugateof {g(x,y)} \ dx dy = -
\int f(z) D^- \conjugateof {g(z)} \ dz \ .
\end{multline*}
In the same way, we get the result for $  \int D^+f(z) \conjugateof {g(z)} \ dz $.
\end{proof}

We define the vector operators on $C^1(\reals^p;\complex)$ by
\begin{equation}\label{eq:D}
D^-= \frac12 \left[  \frac{\partial}{\partial x_j}- \complexu \ \frac{\partial}{\partial y_j} \right]_{j=1,\dots,p} \qquad \textrm{and} \qquad  D^+= \frac12 \left[\frac{\partial}{\partial x_j}+ \complexu \ \frac{\partial}{\partial y_j} \right]_{j=1,\dots,p}
\
\end{equation}
and the components are denoted by $D^-_j$ and $D^+_j$, respectively.

Equation \eqref{eq:complexmonomials} becomes, for a multivariate complex monomial $\prod_{j=1}^p z_j^{n_j} \overline z_j^{m_j}$,

\begin{equation*}
\begin{cases}
D^-\left(\prod_{j=1}^p z_j^{n_j} \conjugateof z_j^{m_j}\right) &= \left[n_k z_k^{n_k-1} \conjugateof z_k^{m_k} \prod_{j\ne k}
z_j^{n_j}\conjugateof z_j^{m_j}\right]_{k} 
\ , \\ \\
  D^+\left(\prod_{j=1}^p z_j^{n_j} \conjugateof z_j^{m_j}\right) &= \left[m_k z_k^{n_k} \conjugateof z_k^{m_k-1} \prod_{j\ne k}
z_j^{n_j} \conjugateof z_j^{m_j}\right]_{k}
 \ .
\end{cases}
\end{equation*}

\subsection{Complex Normal}\label{sec:complex-Normal}

The Complex Normal Distribution is the distribution induced in $\complex$ by the identification $(x,y) \mapsto x+\complexu y$ applied two a couple of independent identically distributed normal $X$ and $Y$. The integration with respect to the image of the Lebesgue measure in $\complex$ is denoted by $dz$. As $x^2+y^2 = z \cdot z = \Re(z\conjugateof z) = z \conjugateof z$ the density of the complex random variable $Z = X + {\mathbf i} Y$ is
\begin{equation*}
  \varphi(z) = \frac1{2\pi\sigma} \expof{-\frac1{2\sigma^2} z \overline z} \ .
\end{equation*}
The complex variance is
$\gamma = \expectof{Z \conjugateof Z} = \expectof{\avalof X ^ 2+ \avalof Y ^ 2} = 2\sigma^2$
and we write $Z  \sim \cnormalof 1 {\gamma}$. Notice that the unit complex variance  obtains when $\sigma^2 = 1/2$.

If $Z \sim \cnormalof 1 {\gamma}$ and $\alpha \in \complex$, then the complex variance of $(\alpha z)$ is $\expectof{(\alpha Z) \conjugateof{\alpha Z}} = \alpha \conjugateof \alpha \gamma$ and it is easy to check that $\alpha Z \sim \cnormalof 1 {\alpha \conjugateof \alpha \gamma}$.
Similarly, $\beta \conjugateof Z \sim \cnormalof 1 {\beta \conjugateof \beta \gamma}$, so that the conjugate $\conjugateof Z$ of a Complex Normal is a Complex Normal with the same variance.

From Proposition~\ref{pr:int_part} with $g=\varphi$,   it follows
\begin{align}\label{eq:Dminusadjoint}
  \int D^-f(z) \ \varphi(z)  \ dz &= - \int f(z) \ D^- \varphi(z)  \ dz  \\
\label{eq:Dplusadjoint}  \int D^+f(z)  \ \varphi(z) \ dz & = - \int f(z) \ D^+\varphi(z) \ dz \ .
\end{align}

Consider $p$ independent standard Complex Normal random variables, $U_j$, $j=1,\dots,p$. Let $C = [c_{i,j}]$ be a $p\times p$ complex matrix. Define the random variable
\begin{equation*}
 Z = CU, \quad U = (U_1,\dots,U_p) \ .
\end{equation*}

By definition, the distribution of $Z$ is a multivariate Complex Normal and it is denoted by $\cnormalof p \Sigma$, with $\Sigma = CC^*$. In the following, we restrict to a non singular matrix $C$, equivalently, to a non singular matrix $\Sigma$.

The density of $U$ is $\phi(u) = \frac1{\pi^p}\expof{-u^*u}$. The density of $\cnormalof p \Sigma$ is obtained by performing the change of variable $U = C^{-1}Z$ in $\reals^{2p}$ and gives
\begin{equation}\label{GausS^N_distr}
  \varphi(z;\Sigma) = \frac1{\pi^p \det \Sigma} \expof{-z^*\Sigma^{-1}z} \ ,
\end{equation}
see \citet{Goodman1963}.

\section{Complex moments of the MCN}
\label{sec:complexmoments}
We come now to the study of complex moments of $\cnormalof p \Sigma$. Let us first introduce some notations. We denote by  $A_{k\cdot}$  and  by $A_{\cdot k}$   the $k$-th row and the $k$-column of the matrix $A$, and by $(e_j)_{j=1,\dots,2p}$ be the canonical basis of $\reals^{2p}$.

\begin{definition}
  Let $\alpha=(n_1,m_1,\dots,n_p,m_p) \in \mathbb Z^{2p}_{\ge 0}$ be a multi-index.  Let $\prod_{j=1}^p z_j^{n_j} \ \conjugateof z_j^{m_j}$ be the corresponding complex monomial.
We denote by $\nu(\alpha)$  the \emph{$\alpha$-moment} of $\cnormalof p \Sigma$,
\begin{equation*}
  \nu(\alpha)=\frac{1}{\pi^p \det (\Sigma)} \int_{\mathbb C^p} z_1^{n_1}\  \overline z_1^{m_1}\ z_2^{n_2}\  \overline z_2^{m_2} \cdots z_p^{n_p} \ \overline z_p^{m_p}  \exp(- z^* \Sigma^{-1} z) \ dz_1 \cdots dz_p
\end{equation*}
and we call \emph{elementary moment} each element of the matrix $\Sigma$:
\begin{equation*}
\sigma_{hk}=\nu(\beta_{hk}) \qquad \mbox{ where } \qquad \beta_{hk}=  e_{2h-1}+e_{2k} \quad h,k = 1,\dots,p \ .
\end{equation*}
Let $N=\left\{ h \in \{1,\dots,p\} \ | \ n_h \neq 0 \right\}$ and $M=\left\{ k \in \{1,\dots,p\} \ | \ m_k \neq 0\right\}$ be the sets of non null indices, whose cardinality is denoted by $\#N$ and $\#M$, respectively. We call $N$ and $M$ the \emph{supporting sets} of the moments $\nu(\alpha)$.

Given $h \in N$ and $k \in M$, we denote by $\alpha^-_{hk}$ the multi-index $(n_1,m_1,\dots,n_h-1,m_h,\dots,n_k,m_k-1,\dots,n_p,m_p)$,  that is
$$\alpha^-_{hk}= \alpha-\beta_{hk} \ .$$
\end{definition}

For instance, if $p=3$, the  elementary moments are
$\sigma_{11}=\nu(1,1,0,0,0,0)$, $ \sigma_{12}=\nu(1,0,0,1,0,0) $, $\sigma_{13}=\nu(1,0,0,0,0,1)$, \dots .

\subsection{Recurrence relations of the moments}

We first extend Eq.s~\eqref{eq:Dminusadjoint} and \eqref{eq:Dplusadjoint} to a generic multivariate complex density \eqref{GausS^N_distr}.

\begin{proposition}\label{pr:derivatives}
Let $D^-$ and  $D^+$ be defined as in Eq. \eqref{eq:D} and define $g(z) = z^*\ \Sigma^{-1}\ z$. Then the following relations hold:
  \begin{equation*}
 \conjugateof z = \Sigma^t D^- g  \quad \text{and} \quad z= \Sigma D^+ g \ .
\end{equation*}
\end{proposition}
\begin{proof}
We have
\begin{equation*}
  \partiald {x_s} z = e_s \quad \text{and} \quad \partiald {y_s} z = \complexu e_s \ .
\end{equation*}
As the directional derivative of $g$ in the direction $r$ is $d_r g(z) = r^*\Sigma^{-1} z + z^* \Sigma^{-1} r$, then
\begin{equation*}
\partiald {x_s} g(z) = \adjointof {e_s}\Sigma^{-1} z + \adjointof{z} \Sigma^{-1} e_s \quad \text{and} \quad   \partiald {y_s} g(z) = \adjointof{(\complexu e_s)} \Sigma^{-1} z + \adjointof{z} \Sigma^{-1} (\complexu e_s) \ ,
\end{equation*}
and we have for each $s = 1,\dots,p$ that
\begin{multline*}
  D^-_k g(z) = \frac12 \left(\adjointof {e_s}\Sigma^{-1} z + \adjointof{z} \Sigma^{-1} e_s\right) - \frac{\complexu}2\left(\adjointof{(\complexu e_s)} \Sigma^{-1} z + \adjointof{z} \Sigma^{-1} (\complexu e_s)\right) = \\
  \frac12 \left(\transposedof{e_s}\Sigma^{-1} z + \adjointof{z} \Sigma^{-1} e_s - \transposedof {e_s}\Sigma^{-1} z + \adjointof {z} \Sigma^{-1} e_s\right) = \adjointof{z} \Sigma^{-1} e_s \ .
\end{multline*}
It follows that
\begin{equation*}
  \transposedof{\left(D^-g(z)\right)} \Sigma = \adjointof{z} \Sigma^{-1} \Sigma = \adjointof z \ ,
\end{equation*}
whose transposed equation is the first equality to be proved.

For the other equality, we observe that for each $s = 1,\dots,p$ we have
\begin{multline*}
  D^+_s g(z) = \frac12 \left(\adjointof {e_s}\Sigma^{-1} z + \adjointof{z} \Sigma^{-1} e_s\right) + \frac{\complexu}2\left(\adjointof{(\complexu e_s)} \Sigma^{-1} z + \adjointof{z} \Sigma^{-1} (\complexu e_s)\right) = \\
  \frac12 \left(\transposedof{e_k}\Sigma^{-1} z + \adjointof{z} \Sigma^{-1} e_s + \transposedof {e_s}\Sigma^{-1} z - \adjointof {z} \Sigma^{-1} e_s\right) = \transposedof{e_s} \Sigma^{-1} z \ ,
\end{multline*}
hence $\Sigma D^+ g(z) = \Sigma \Sigma^{-1} z = z$.
\end{proof}

Previous proposition  allows us to extend the integrations by parts of Eq.~\eqref{eq:Dminusadjoint}  to the case of a MCND. In fact, for the density in Eq.~\eqref{GausS^N_distr}, we have
\begin{equation*}
  \transposedof \Sigma D^- \varphi(z;\Sigma) = -\conjugateof z \varphi(z;\Sigma) \quad \text{and} \quad   \Sigma D^+ \varphi(z;\Sigma) = -z \varphi(z;\Sigma).
\end{equation*}

If we multiply each one of the equations above by a monomial, the monomial itself is multiplied by $\conjugateof z$ and $z$, respectively, in the right hand sides, hence integration produces a relation between moments, increasing the degree by 1. Conversely, we find, for each  $N$ and $M$, a linear expression in the moments of degree reduced by 2. Notice that the recurrence is clearly finite. The argument is formalised in the following theorem.

\begin{proposition}[Recurrence relations for the moments]\label{th:recurrence}
Given the multi-index $\alpha$ with supporting sets $N$ and $M$, there are $\#N + \#M \le 2p$ recurrence relations for the moment $\nu(\alpha)$,
\begin{equation}\label{eq:recurrence-relations}
\nu(\alpha)=\sum_{k \in M} m_k \ \sigma_{hk}\ \nu(\alpha^-_{hk}), \ h\in N, \qquad \mbox{and} \qquad
\nu(\alpha)=\sum_{h \in N} n_h \ \sigma_{hk}\ \nu(\alpha^-_{hk}), \ k \in M \ .
\end{equation}
\end{proposition}

\begin{proof}
Let $f(z,\alpha)= \prod_{j=1}^p z_j^{n_j}\  \overline z_j^{m_j}$ be the complex monomial with exponent $\alpha$, and let be given $r \in N$.
\begin{eqnarray*}
  \nu(\alpha)&=&\int_{\mathbb C^p} f(z,\alpha) \varphi(z; \Sigma) dz =
    \int_{\mathbb C^p} f(z,\alpha-e_{2h-1}) z_h \varphi(z; \Sigma) dz   \\
&=& -   \int_{\mathbb C^p} f(z,\alpha-e_{2h-1}) \Sigma_{r.} D^+\varphi(z; \Sigma) dz\\
&=&- \sum_{k=1}^p \sigma_{hk}   \int_{\mathbb C^p} f(z,\alpha-e_{2h-1}) D_k^+\varphi(z; \Sigma) dz\\
  &=&\sum_{k\in M} \sigma_{hk}   \int_{\mathbb C^p} D^+_kf(z,\alpha-e_{2h-1})\varphi(z; \Sigma) dz\\
  &=& \sum_{k\in M} \sigma_{hk} m_k  \int_{\mathbb C^p} f(z,\alpha-e_{2h-1}-e_{2k})\varphi(z; \Sigma) dz
  = \sum_{k\in M} \sigma_{hk} m_k \nu(\alpha_{rk}^-)  \ .
  \end{eqnarray*}
Analogously, by considering $\overline z_k$ instead of $z_h$,
$\nu(\alpha)=\sum_{h\in N} n_h \ \sigma_{hk}\ \nu(\alpha^-_{hk})$ for each $k \in M$.
\end{proof}

The recurrence relations \eqref{eq:recurrence-relations} allow us to compute any complex moment as a linear function of moments whose total degree is smaller that the given one by 2. Hence, each complex moment is a polynomial of the elementary moments $\sigma_{hk}$, $h \in N$, $k \in M$. The following example shows a simple case of the recurrence.

 \begin{example}\label{pr:rec-mom}
If $p=2$, the  recurrence relations of Proposition~\ref{th:recurrence} are the following.
\begin{multline*}
\nu(n_1,m_1,n_2,m_2) =
m_1\ \nu(1,1,0,0)\ \nu(n_1-1,m_1-1,n_2,m_2) + \\ m_2\ \nu(1,0,0,1)\ \nu(n_1-1,m_1,n_2,m_2-1)
\nu(n_1,m_1,n_2,m_2) = \\
m_2\ \nu(0,0,1,1)\ \nu(n_1,m_1,n_2-1,m_2-1)+m_1\ \nu(0,1,1,0)\ \nu(n_1,m_1-1,n_2-1,m_2) = \\
n_1\ \nu(1,1,0,0)\ \nu(n_1-1,m_1-1,n_2,m_2)+n_2\ \nu(0,1,1,0)\ \nu(n_1,m_1-1,n_2-1,m_2) = \\
n_2\ \nu(0,0,1,1)\ \nu(n_1,m_1,n_2-1,m_2-1)+n_1\ \nu(1,0,0,1)\ \nu(n_1-1,m_1,n_2,m_2-1) \  .
\end{multline*}
\end{example}

When some covariances are zero, the corresponding terms in Equations \eqref{eq:recurrence-relations} are zero. We derive now a modified form of the recurrence with non-zero terms only.

\begin{definition}\label{def:Si1}
Consider the bipartite graph with vertices $N \cup M$ and undirected edges $\set{h,k}$ with $h \in N$, $k\in M$, and $\sigma_{hk} > 0$. In turn, the bipartite graph defines subsets of $N$ and $M$ by considering the neighborhoods of each vertex.  We denote by $S^N_h$, $h \in N$, and $S^M_k$, $k \in M$, respectively, the sets
$$ S^N_h = \{ k \in M \ \vert \ \sigma_{hk} \neq 0 \}  \qquad \textrm{and} \qquad S^M_k = \{ h \in N \ \vert \ \sigma_{hk} \neq 0 \} \ . $$
\end{definition}

\begin{definition}\label{def:Si2}
Consider the graph whose nodes are the elements of $M$, and there is an edge between two nodes if and only if both nodes belong to the same $S^N_h$ for some $h \in N$. The connected  components of this graph define a partition of $M$ that we call the \emph {partition induced} by the $S^N_h$'s. The partition on $N$ induced by the $S^M_k$ is defined analogously.
\end{definition}

Notice that the partitions of $M$  and $N$ are uniquely defined and they can  be the trivial partition.


\begin{corollary}\label{cor:recurrence-relations-sparse}
The recurrence relations of the moments in Proposition~\ref{th:recurrence} can be written as:
\begin{equation} \label{eq:recurrence-relations-sparse}
\nu(\alpha)=\sum_{k \in S^N_h} m_k \ \sigma_{hk}\ \nu(\alpha^-_{hk }), \ h\in N, \qquad \mbox{and} \qquad
\nu(\alpha)=\sum_{h \in S^M_k} n_h \ \sigma_{hk}\ \nu(\alpha^-_{hk}), \ k \in M \ .
\end{equation}
\end{corollary}

\begin{proof}
For example, in the first case, if $k \in M \setminus S_h^N$, then $\sigma_{hk}=0$ and that term is missing in the RHS of the recurrent relations.
\end{proof}

Now we write the system of the previous recurrence relations in matrix form.
\begin{definition}\label{def:mat_rec}
For each $h \in N$ such that $S_h^N \neq \emptyset$, let $A_h$ be  the $2p \times \# S_h^N$ matrix  with columns
$$a_{(hk)}=\sigma_{hk}(m_k e_h+n_he_{p+k}) \qquad \textrm {for } k\in S_h^N \ ,$$
and let $A$ be the row block matrix $A=[A_h]_{h \in N, S^N_h\neq \emptyset}$. Moreover,  let $t_h$ be  the $\# S_h^N$ vector
$$t_h=\left (\nu(\alpha_{hk}^-) \right)_{k \in S_h^N \neq \emptyset} \ , $$
and let $t$ be the column block vector $t=(t_h)_{h \in N}$.

Let $b$ be the $2p$ vector:
$$b = \sum_{h\in N} e_h + \sum_{k \in M} e_{p+k} \ . $$
\end{definition}

\begin{corollary}
The system of the recurrence relations in Corollary~\ref{cor:recurrence-relations-sparse} can be expressed in matrix form as:
\begin{equation}\label{eq:A}
A \ t  = \nu(\alpha)\ b \ .
\end{equation}

If $\#N <p$ then, for each $h \notin N$, the $h$-th row of $A$ is null and $b_{h} = 0$, so that the $h$-th equation is $0=0$.
Analogously,  if $\#M <p$ then, for each $k \notin M$, the $(p+k)$-th row of $A$ is null and $b_{p+k} = 0$, so that the $(p+k)$-th equation is $0=0$.
\end{corollary}
\begin{proof}
The matrix form of the recurrence relations is easily derived. If $h \notin N$, then $b_h=0$. Furthermore, $e_h$ is not involved in the construction of any column of $A$, so that the $h$-th element of each column is zero.
Analogously for $k \notin M$.
\end{proof}

\begin{example}
If $N=M=\{1,\dots,p\}$ and all the elements of $\Sigma$ are different from zero, the $2p \times p^2$ matrix $A=[A_1 \vert \dots \vert A_p]$ has the form:
{\footnotesize{\begin{equation*}
 \left[\begin{array}{cccc|cccc|c|cccc}
    m_1\sigma_{11}&m_2\sigma_{12} &\dots&m_p\sigma_{1p}&   0          & 0    & 0    & 0            &\dots & 0            & 0    & 0    & 0\\
    0             &0     &0    & 0            &m_1\sigma_{21}&m_2\sigma_{22}&\dots &m_p\sigma_{2p}&\vdots&0             &0 &0 &0 \\
   \vdots         &\dots &\dots& \vdots       & \vdots       &\dots &\dots &\vdots        &\vdots&\vdots        &\dots &\dots &\vdots\\
    0             &0     &0    & 0            &  0           & 0    & 0    & 0            &\dots &m_1\sigma_{p1}&m_2\sigma_{p2} &\dots &m_p\sigma_{pp} \\ \hline
    n_1\sigma_{11}&0     &0    & 0           &n_2\sigma_{21}& 0    &0      &0             &\dots &n_p\sigma_{p1}&0     &0 &0  \\
    0             &n_1\sigma_{12}& 0   &  0           & 0   &n_2\sigma_{22}& 0     &  0           &\ddots& 0            &n_p\sigma_{p2}&  0&0 \\
    0             &0     &\ddots    &  0      & 0           &  0           &\ddots& 0            &\ddots&  0 &0&\ddots&0\\
    0             &0     &0    &n_1\sigma_{1p}& 0           & 0    & 0    &n_2\sigma_{2p}&\dots & 0            &  0 &  0   &n_p\sigma_{pp}
  \end{array} \right]\ .
\end{equation*}}}
\end{example}
We observe that by suitable permutations of the rows and the columns of the matrix $A$ we obtain a matrix with the same structure and the $m$'s and the $n$'s  switched. Then Definition~\ref{def:mat_rec} can be expressed using the sets $S_k^M$, $k \in M$.

\subsection{Null moments}

The $\# N + \#M$ Eq.s in Proposition~\ref{th:recurrence} all give the value of the $\alpha$-moment. We present some necessary conditions for a moment $\nu(\alpha)$ to be zero.

We consider  the following linear combination of the recurrences in Proposition~\ref{th:recurrence}
\begin{equation*}
\left ( \sum_{h \in N } n_h  \nu(\alpha)- \sum_{k  \in M} m_k \nu(\alpha)  \right ) =
\sum_{h\in N} n_h \left (\sum_{k \in M} m_k \ \sigma_{hk}\ \nu(\alpha^-_{hk})\right) -\sum_{k \in M} m_k\left (\sum_{h \in N} n_h \ \sigma_{hk}\ \nu(\alpha^-_{hk}) \right )
 \end {equation*}
and so
$$  \nu(\alpha)  \left(  \sum_{h \in N } n_h -  \sum_{k \in M } m_k  \right ) = 0  \ . $$
We conclude that, if  $ \sum_{h \in N } n_h -  \sum_{k \in M } m_k \neq 0$, then $\nu(\alpha) =0$.
In presence of a non-trivial induced partition of the supporting sets (Definition~\ref{def:Si2}), we can prove a further necessary condition for the nullity of the moment, see second item of the following Theorem.

\begin{theorem}\label{th:conditions}
Let $\alpha$ be a multi-index and let $S^N_h$, $h \in N$, be as in Definition~\ref{def:Si1}. The moment $\nu(\alpha)$ is null if one of the following conditions hold.
\begin{enumerate}
\item There exists $h \in N$ such that $S^N_h = \emptyset$ or there exists $k \in M$ such that $S^M_k = \emptyset$.
\item  All the sets $S^N_h \neq \emptyset$, $h \in N$, and there exists a partition  $\{ M_1,\dots,M_q\}$ of $M$  induced by $S^N_h$'s such that
$$\exists h \in \{1,\dots,q\} \qquad \sum_{i \vert S^N_i \subset M_h } n_i\ne \sum_{j\in M_h} m_j \ .$$
\end{enumerate}
\end{theorem}
\begin{proof}
\begin{enumerate}
\item If $ \exists \ h \in N$ such that $S^N_h = \emptyset$ then the  recurrence relation
$$ \nu(\alpha) =\sum_{k \in M} m_k \sigma_{hk} \nu(\alpha_{hk}^-)$$
consists of null addends,  since $\sigma_{hk}=0$, $k \in M$. Analogously for $k \in M$.
 \item Assume that the  sets $S^N_h \neq \emptyset$, $h \in N$. For each element $M_r$ of the partition, we consider the  linear combination of the equations in~\eqref{eq:A} whose coefficients are the elements of the vector $c_r=\sum_{i \vert S^N_i \subset M_r } n_i e_i -\sum_{j \in  M_r} m_j e_{p+j} $, that is, in matrix form, $c_r^t \ A  t = \nu(\alpha) \ c_r^t b$. By trivially computation we have
 $$\nu(\alpha) c_r^t b=   \nu(\alpha) \left ( \sum_{i \vert S^N_i \subset M_r  } n_i - \sum_{j  \in M_r} m_j  \right ) \ .$$
 Consider the scalar product of $c_r$ with a columns $a_{(hk)}$ of $A_h$, $k \in S^N_h$.
Notice that or $S^N_h$ is a subset of $M_r$ or $S^N_h \cap M_r = \emptyset$, since the partition of $M$ is induced by the $S^N_h$'s.

If $S^N_h \subset M_r$:
 $$c_r^t a_{(hk)}=\left( \sum_{i \vert S^N_i \subset M_r } n_i e_i -\sum_{j \in  M_r} m_j e_{p+j} \right )^t \left( \sigma_{hk} (m_k e_h  + n_h e_{p+k})\right )   = \sigma_{hk} (n_h m_k  - m_k n_h) =0 \ . $$
If $S^N_h \cap  M_r =\emptyset$:
 $$c_r^t a_{(hk)}=\left( \sum_{i \vert S^N_i \subset M_r } n_i e_i -\sum_{j \in  M_r} m_j e_{p+j}\right ) ^t \left (\sigma_{hk} (m_k e_h  + n_h e_{p+k}) \right )  =0 \ , $$
since $e_h$  and $e_{p+k}$ do not generate $c_r$.

Then, for each $M_r$ it holds
$$ \nu(\alpha) \left ( \sum_{i \vert S^N_i \subset M_r } n_i - \sum_{j  \in M_r} m_j  \right ) = c_r^t b = c_r^t A \ t = 0       $$
and the thesis follows.

 \end{enumerate}
\end{proof}

\begin{remark}\label{rm:equiv}
Item~1 of the previous theorem can be expressed as:
$$ \bigcup_{k\in M} S^M_k \neq N \qquad \textrm{or} \qquad \bigcup_{h \in N} S^N_h \neq M \ .$$
In fact
$$\exists h \in N \textrm{  s.t.  } S^N_h = \emptyset \qquad \Longleftrightarrow \qquad   \bigcup_{k\in M} S^M_k \neq N  \  $$
because, if $ S^N_h=\emptyset$, then $\sigma_{hk}=0$ for all $k \in M$. So that each $S^M_k$ does not contain the index $h \in N$ and the thesis follows.  Analogously, the vice versa can be shown.

\end{remark}

\begin{example}\label{ex:main}
We consider $p=5$ not independent complex multivariate normal variables with the matrix $\Sigma$ such that
$$\sigma_{13}=\sigma_{14}=\sigma_{24}=\sigma_{25}=\sigma_{35}=\sigma_{45}=0 \ . $$
With such a matrix $\Sigma$ we show different cases depending on the choice of the sets $N$ and $M$.

\medskip
\noindent\emph {Case a)}
If $N=M=\{1,\dots,5\}$, the matrix $A$ corresponding to the previous matrix $\Sigma$  is:
{\scriptsize{\begin{equation*}
 \left[\begin{array}{ccc|ccc|ccc|cc|cc}
    m_1\sigma_{11}&m_2\sigma_{12} &m_5\sigma_{15}&   0          & 0    & 0    & 0            &0 & 0            & 0    & 0    & 0 & 0\\
    0             &0     &0        &m_1\sigma_{21}&m_2\sigma_{22}&m_3\sigma_{23}&0&0 &0&0        &0 &0 &0 \\
    0             &0     &0    & 0            &  0           & 0    &m_2\sigma_{32}&m_3\sigma_{33} &m_4\sigma_{34} & 0 & 0 & 0 & 0 \\
    0             &0     &0    & 0            &  0           & 0   & 0 & 0 & 0  &m_3\sigma_{43}&m_4\sigma_{44} &0 & 0  \\           0             &0     &0    & 0            &  0           & 0   & 0 & 0 & 0  & 0 & 0 &m_1\sigma_{51}&m_5\sigma_{55}  \\       \hline
    n_1\sigma_{11}&0     & 0           &n_2\sigma_{21} &0      &0     & 0 & 0 &0     &0 &0 &n_5\sigma_{51} &0     \\
  0 & n_1\sigma_{12}&0     & 0           &n_2\sigma_{22} &0      &n_3\sigma_{32}      & 0 & 0 &0     &0 &0 &0     \\
  0 &  0 &0    & 0           & 0& n_2\sigma_{23} &0      &n_3\sigma_{33}      & 0 &n_4\sigma_{43}& 0 & 0 &0     \\
  0 &  0 &0     & 0           & 0&  0 &  0 &0    & n_3\sigma_{34} &0      &n_4\sigma_{44}      & 0 &0     \\
  0 &  0 &n_1\sigma_{15}     & 0           & 0&  0 &  0 &0    & 0&0      &0      & 0 &n_5\sigma_{55}     \\
  \end{array} \right]\ .
\end{equation*}}}

\medskip
\noindent\emph {Case b)}
If   $N=\{1,2,3,4,5\}$, $M=\{1,4,5\}$, we have $S^N_1=\{1,5\}$, $S^N_2 = \{1\}$, $S^N_3=\{4\}$, $S^N_4=\{4\}$, $S^N_5=\{1,5\}$, and there exists the induced partition of $M$ given by $M_1=S^N_1\cup S^N_2 \cup S^N_5=  \{1,5\}$ and $M_2=S^N_3 \cup S^N_4= \{4\}$. In such a case the matrix $[\sigma_{hk}]_{\substack{h \in N,\\ k\in M}}$ is
\begin{eqnarray*}
[\sigma_{hk}]_{\substack{h \in N,\\ k\in M}}=\left [\begin{array}{ccc}
\sigma_{11} &0  & \sigma_{15} \\
\sigma_{21} &  0 & 0\\
0 & \sigma_{34} & 0 \\
 0  &  \sigma_{44} & 0 \\
\sigma_{51} & 0 &  \sigma_{55}
\end{array}\right ] =  \ \textrm{(by row/column permutations)}  \ =
\left [\begin{array}{ccc}
\sigma_{11}  & \sigma_{15} &0  \\
\sigma_{21} &  0 & 0\\
\sigma_{51} &  \sigma_{55} &0\\
0 & 0 & \sigma_{34}  \\
 0  & 0 & \sigma_{44}  \end{array}\right ]  \ .
\end{eqnarray*}
The permuted matrix highlights the induced partition. \\
In this case the conditions in Item 2 of Theorem~\ref{th:conditions} are $n_1+n_2+n_5 \neq m_1+m_5$ or $n_3+n_4 \neq m_4$.

\medskip
\noindent\emph {Case c)}
If  $N=\{1,2,3,4,5\}$ and $M=\{2,3,4\}$ then $\nu(\alpha)=0$, for all $\alpha$,  since  $S^N_5 = \emptyset$.
\end{example}

\subsection{Computation of the moments}

The form of the recurrence relation for the moments suggests that $\nu(\alpha)$ is a linear combination of $\nu(\beta_{hk})$ with coefficients which are themselves monomials in the covariances.  In the simple case when $\alpha = c  \ \beta_{hk}$, $c \in \mathbb Z_{\ge}$,
$$ \beta_{hk} = e_{2h-1} + e_{2k}\ , \qquad h \in N, \  k \in M \ ,$$
each recurrence relation contains one term only, hence only one reduction is feasible, so that
$$\nu(c \ \beta_{hk}) = c! \  \nu(\beta_{hk})^c \ .$$

In general, the value of $\nu(\alpha)$, with $\alpha$ such that $\sum_{h \in N} n_h = \sum_{k \in M} m_k$, can be obtained considering that $\alpha$ is generated by the vectors with integer coefficients:
$$\alpha= \sum_{\substack{h\in N, \ k \in M}} a_{hk} \ \beta_{hk}\quad \textrm{with  } \quad a_{hk} \in \mathbb Z_{\ge 0}\ ,$$
and there is a tree of reduction path. As a consequence, the coefficient vector $a=[a_{hk}]_{\substack{h\in N \\ k \in M}}$ is not uniquely determined as it was in the simple case above. For instance, if $\alpha = (2,1,2,3)$ then
$$ (2,1,2,3)=(1,1,0,0)+(1,0,0,1) +2(0,0,1,1) = 2(1,0,0,1) +(0,1,1,0) + (0,0,1,1) \ .$$

Considering all the possible coefficient vectors $a$ that produce the same $\alpha$, it is useful to define the subset $I(\alpha) \subset \mathbb Z_{\ge 0}^{p^2}$ associated to each $\alpha$-moments as follows.
\begin{definition}\label{def:Ialpha}
Let $\alpha$ be a multi-index. The set $I(\alpha) \subset \mathbb Z_{\ge 0}^{p^2}$ is given by
\begin{multline}
  I(\alpha) = \{ a=(a_{1,1},\dots,a_{1,p},a_{2,1},\dots  a_{p,p}) \in \mathbb Z_{\ge 0}^{p^2}  \;|\;  \\
  l_{ij}(\alpha,a) \le a_{ij} \le L_{ij}(\alpha,a) , i,j=1,\dots, p \} \ ,
\end{multline}
where the bounds are
\begin{eqnarray}\label{def:bounds}
l_{ij}(\alpha,a) &=& 0 \wedge \left( \sum_{h=1}^ i n_h - \sum_{h=j+1}^ p m_h - \sum_{\substack{h=1,\dots, i; \; k=1, \dots, j \\ (h,k) \neq (i,j)}} a_{hk}\right)\nonumber \\
L_{ij}(\alpha,a) &=& \left (n_i - \sum_{h=1}^{j-1} a_{ih} \right ) \vee \left(m_j- \sum_{h=1}^{i-1} a_{hj} \right) \ .
\end{eqnarray}
\end{definition}

The following Theorem gives the expression of  the $\alpha$-moment of  $\cnormalof p \Sigma$.
The proof is based on several lemmas and propositions given in the Appendix.

\begin{theorem}\label{th:main}
Let $\alpha$ be a multi-index and let $S^N_h$, $h \in N$, be as in Definition~\ref{def:Si1}.
Assume that for all the elements of the  partition  $\{ M_1,\dots,M_q\}$ of $M$  induced by $S^N_h$'s we have
$$ \sum_{i \vert S^N_i \subset M_h } n_i = \sum_{j\in M_h} m_j  \quad \forall \ h=1,\dots,q \ .$$
Then the $\alpha$-moment of  $\cnormalof p \Sigma$ is
\begin{eqnarray}\label{Normal_moment}
\nu(\alpha)= \sum_{a \in I(\alpha)} \Pi(\alpha,a)
\end{eqnarray}
where  $\Pi(\alpha,a)$ is
\begin{equation}\label{def_pi}
\Pi(\alpha,a) = \prod_{h=1}^p n_h! m_h! \prod_{h,k=1}^{p}\frac{\sigma_{hk}^{a_{hk}}}{a_{hk}!}
\end{equation}
by setting $\sigma_{hk}^0=1$ also when $\sigma_{hk}=0$. The set $I(\alpha)$ is  as in Definition~\ref{def:Ialpha}.
\end{theorem}
\begin{proof}
First, we show that  $\sum_{a \in I(\alpha)}\Pi(\alpha,a) $  is the elementary moment $\sigma_{hk}$ when $\alpha=\beta_{hk}=  e_{2r-1}+e_{2s}$. Then we show that $\sum_{a \in I(\alpha)}\Pi(\alpha,a) $
satisfies the recurrence relations for a general $\alpha$, so that the thesis follows.

\bigskip \noindent
{\em First part.}\\ Let $\alpha=\beta_{hk}$. Since $n_h=m_k=1$ and $n_i=m_j=0$ for each $i\neq h$  and $j\neq k$, Corollary~\ref{cor:null_exp} implies that $a_{ir}=a_{rj} = 0$, for  $i\neq h$, $j\neq k$ and $r\le p$, that is
$a_{hk}$ is the unique element of the vector $a$ different from zero.
Furthermore $a_{hk}=1$ since $l_{hk} = 0 \wedge (\sum_{r=1}^h n_h -\sum_{r=k+1}^p m_r - a_{11}-\dots - a_{h,(k-1)}) =1$ and  $L_{hk}= (n_h-\sum_{r=1}^{k-1} a_{hr}) \vee (m_k-\sum_{r=1}^{h-1} a_{rk} )= 1 \vee 1 =1$.
We obtain that
\begin{equation*}
\sum_{a \in I(\alpha) }\Pi(\alpha,a) = n_h!m_k! \frac{\sigma_{hk}^{a_{hk}}}{a_{hk}!} \ = \ \sigma_{hk} \ .
\end{equation*}

\noindent
{\em Second part.} \\
Let us consider $\sum_{k=1}^pm_k \sigma_{hk}  \nu(\alpha_{hk}^-)$. The proof is analogous if we consider $\sum_{h=1}^p n_h \sigma_{hk} \nu(\alpha_{hk}^-)$.

We remember that   $ \alpha_{hk}^-$ is  the vector $\alpha$ containing the values $n_h-1$ and $m_k-1$ instead of $n_h$ and $m_k$: $\alpha_{hk}^- = \alpha - e_{2h-1} - e_{2k}$.
Furthermore, for each pair $(h,k)\in \{1,\dots, p\}^2$  let $ a_{hk}^+ $ be the vector  $a$ containing the value $a_{hk}+1$  instead of $a_{hk}$:
$  a_{hk}^+ = a+ e_{(h-1)p+k} $. \\
From Equation~\eqref{def_pi} we have
\begin{multline}
\sigma_{hk}\Pi(\alpha_{hk}^-,a)= (n_h-1)!(m_k-1)!  \frac{\sigma_{hk}^{a_{hk}+1}}{a_{hk}!}\prod_{\substack{r=1\\ r \neq h}}^p n_r! \prod_{\substack{s=1\\ s \neq k}}^p m_s!
\prod_{\substack{r,s=1\\ (r,s) \neq (h,k)}}^p \frac{\sigma_{rs}^{a_{rs}}}{a_{rs}!}  \\
=\frac{\prod_{r=1}^p n_r!  m_r! }{n_h m_k } (a_{hk}+1) \frac{\sigma_{hk}^{a_{hk}+1}}{(a_{hk}+1)!}\prod_{\substack{r,s=1\\ (r,s) \neq (h,k)}}^p \frac{\sigma_{rs}^{a_{rs}}}{a_{rs}!}  =  \frac{(a_{hk}+1)}{n_h m_k} \ \Pi(\alpha, a_{hk}^+) \ . \label{form_1}
\end{multline}
Let $\delta(\alpha) =\sum_{a \in I(\alpha)} \Pi(\alpha_{hk},a)$. From Equations~\eqref{def_pi} and~\eqref{form_1}  it follows
\begin{multline}\label{eq:rec}
 \sum_{k=1}^{p}m_k \sigma_{hk}  \delta(\alpha_{hk}^-) = \sum_{k=1}^{p}m_k \sigma_{hk} \sum_{a \in I(\alpha_{hk}^-)} \Pi(\alpha_{hk}^-,a)= \frac{1}{n_h}
 \sum_{k=1}^{p} \sum_{a \in I(\alpha_{hk}^-)} (a_{hk}+1)  \ \Pi(\alpha,a_{hk}^+) \\
 =  \frac{1}{n_h}
 \sum_{k=1}^{p} \sum_{b \in I_{hk}^+} b_{hk}  \ \Pi(\alpha,b)  \ ,
\end{multline}
where  $b=a_{hk}^+$ and $I^+_{hk}= \left \{ b= a_{hk}^+ \vert a \in I( \alpha_{hk}^-) \right \}$.
From Proposition~\ref{Prop:bound_b} it follows \begin{eqnarray}\label{eq_rs}
\sum_{b \in I_{hk}^+} b_{hk}  \ \Pi(\alpha,b)  = \sum_{b \in I(\alpha)} b_{hk}  \ \Pi(\alpha,b)
\end{eqnarray}
since, if  $b$ assumes  values in $I(\alpha) \setminus I_{hk}^+$, then $b_{hk} =0$. In fact, if $b \in  I(\alpha) \setminus I_{hk}^+$, then  $ b_{hk}  =  0 $ or, for   $i<h$ and $j < k$, $b_{ij}$ assume a value at the end of some interval and so,  from Proposition~\ref{Prop:values} we have $b_{hk}=0$.
Finally, we have
\begin{multline*}
\sum_{k=1}^p m_k \sigma_{hk} \delta(\alpha_{hk}^-) =  \frac{1}{n_h}
 \sum_{k=1}^{p} \sum_{b \in I(\alpha) } b_{hk}  \ \Pi(\alpha,b)  =   \sum_{a \in I(\alpha) }    \frac{ \sum_{k=1}^{p} a_{hk} }{n_h} \ \Pi(\alpha,a)  \ ,
 \end{multline*}
and so, since  Proposition~\ref{extreme} shows $ \sum_{s=1}^{p} a_{hk} =n_h$, $\sum_{k=1}^p m_k \sigma_{hk} \delta(\alpha_{hk}^-) = \delta(\alpha) $. The thesis follows because the function $\delta(\alpha)$ satisfies the recurrence relations and coincides with the elementary moments, so that $\delta(\alpha) = \nu(\alpha)$.
  \end{proof}

\begin{example} Let us consider the case with $p=2$ and $p=3$.\\
\begin{itemize}
\item If $p=2$ and $n_1+n_2=m_1+m_2$, then  $\nu(\alpha)= \sum_{a\in I(\alpha)} \Pi(\alpha,a)$, where
$$ I(\alpha)=\{a_{11} \in \mathbb Z_{\ge 0} \ | \ 0 \wedge (n_1 -m_2) \le a_{11} \le n_1 \vee m_1\}  \qquad \text{and}$$
\begin{equation*}
\Pi(\alpha,a) = n_1! n_2! m_1! m_2! \ \frac{\sigma_{11}^{a_{11}}}{a_{11}!} \frac{\sigma_{12}^{n_1-a_{11}}}{(n_1-a_{11})!} \ \frac{\sigma_{21}^{m_1-a_{11}}}{(m_1-a_{11})!} \ \frac {\sigma_{22}^{n_2-m_1+a_{11} }}{(n_2-m_1+a_{11})!}\ .
\end{equation*}
\item If $p=3$ and $n_1+n_2+n_3=m_1+m_2+m_3$, then  $\nu(\alpha)= \sum_{a\in I(\alpha)} \Pi(\alpha,a)$, where
\begin{eqnarray*}
 I(\alpha) &=& \left \{ \ a=(a_{11},a_{12},a_{21},a_{22} ) \in \mathbb Z_{\ge 0}^4 \ | \ 0\wedge (n_1-m_2-m_3)\le a_{11} \le n_1 \vee m_1, \right. \\
 && 0\wedge (n_1-m_3-a_{11})\le a_{12} \le (n_1-a_{11} )\vee m_2,   \\
&& 0\wedge (n_1+n_2-m_2-m_3-a_{11})\le   a_{21} \le n_2 \vee (m_1-a_{11}) , \\
&& \left .0\wedge (n_1+n_2-m_3-a_{11}-a_{12}-a_{21}) \le a_{22} \le (n_2-a_{21}) \vee (m_2-a_{12}) \right \}
\end{eqnarray*}
 and
 \begin{multline*}
\Pi(\alpha,a) = n_1! n_2! n_3! m_1! m_2! m_3! \ \frac{\sigma_{11}^{a_{11}}}{a_{11}!} \
 \frac{\sigma_{12}^{a_{12}}}{a_{12}!}  \  \frac{\sigma_{21}^{a_{21}}}{a_{21}!} \
\frac{\sigma_{22}^{a_{22}}}{a_{22}!} \
\frac{\sigma_{13}^{n_1-\sum_{h=1}^{2}a_{1h}}}{ (n_1-\sum_{h=1}^{2}a_{1h} )! } \\
\frac{\sigma_{23}^{n_2-\sum_{h=1}^{2}a_{2h}}}{ (n_2-\sum_{h=1}^{2}a_{2h} )! } \
\frac{\sigma_{31}^{m_1-\sum_{h=1}^{2}a_{h1}}}{ (m_1-\sum_{h=1}^{2}a_{h1} )! }\
\frac{\sigma_{32}^{m_2-\sum_{h=1}^{2}a_{h2}}}{ (m_2-\sum_{h=1}^{2}a_{h2} )! } \
\frac{\sigma_{33}^{ n_3-m_1-m_2 + \sum_{h,k=1}^{2}a_{hk}}   } {(  n_3-m_1-m_2 + \sum_{h,k=1}^{2}a_{hk}  )!}
 \end{multline*}
\end{itemize}
\end{example}

\begin{example}
Let $p=5$ and $\alpha= (2, 0, 1, 2, 1, 2, 2, 3, 1, 0)$, where $m_1=m_5=0$. In such a case $ \sum_{h=1}^5 n_h= \sum_{k=1}^5 m_k$, so that the  condition for using formula~\eqref{Normal_moment} is satisfied.
Since $m_1=m_5$, from Corollary~\ref{cor:null_exp} it follows that $a_{h,1} = a_{h,5}=0$, $h=1,\dots, 5$.
The exponents are such that:
\begin{equation*}
\begin{array}{lll}
0 \le a_{12}\le 2 &0 \le  a_{13}\le  2-a_{12} & a_{14}= 2-a_{12}-a_{13} \\
0 \le a_{22}\le 1 \vee (2-a_{12}) & 0 \le a_{23}\le (1-a_{22}) \vee (2-a_{13})& a_{24}= 1-a_{22}-a_{23} \\
0 \le a_{32} \le  1 \vee (2-a_{12}-a_{22})& 0 \le a_{33}\le (1-a_{32}) \vee (2-a_{13}-a_{23}) & a_{34}= 1-a_{32}-a_{33} \end{array}
\end{equation*}
\vskip-0.55cm
\begin{equation*}
 \begin{array}{l}
0 \le a_{42}\le 2 \vee (2-a_{12}-a_{22}-a_{32}) \\
  0 \wedge (1-a_{12}-a_{13}-a_{22}-a_{23} -a_{32}-a_{33}-a_{42}) \le a_{43}\le (2-a_{42}) \vee (2-a_{13}-a_{23}-a_{33})\\
 a_{44}= 2-a_{42}-a_{43}
\end{array}
\end{equation*}\vskip-0.55cm
\begin{equation*}
 \begin{array}{lll}
a_{52}=2-a_{12}-a_{22}-a_{32}- a_{42} & a_{53} = 2-a_{13}-a_{23} -a_{33}-a_{43} \\
 a_{54}= -3 + a_{12} + a_{13} +a_{22} + a_{23} + a_{32} + a_{33} + a_{42} + a_{43}
\end{array}
\end{equation*}
As in Example~\ref{ex:main}, we consider a matrix $\Sigma$ such that
$$\sigma_{13}=\sigma_{14}=\sigma_{24}=\sigma_{25}=\sigma_{35}=\sigma_{45}=0 \ . $$
The multi-index $\alpha$ belongs to the \emph{Case c)} in Example~\ref{ex:main} and so we know that $\nu(\alpha)=0$. Such a result can be also obtained using  Equation~\eqref{Normal_moment}:
\begin{eqnarray*}
\sum_{a \in I(\alpha)} &&
\sigma_{11}^0 \ \frac{\sigma_{12}^{a_{12}}}{a_{12}!} \ \frac{\boldsymbol{\sigma}_{13}^{a_{13}}}{a_{13}!} \ \frac{\boldsymbol{\sigma}_{14}^{2-a_{12}-a_{13}}}{(2-a_{12}-a_{13})!} \ \sigma_{15}^0 \quad
\sigma_{21}^0 \ \frac{\sigma_{22}^{a_{22}}}{a_{22}!} \ \frac{\sigma_{23}^{a_{23}}}{a_{23}!} \ \frac{\boldsymbol{\sigma}_{24}^{1-a_{22}-a_{23}}}{(1-a_{22}-a_{23})!} \ \boldsymbol{\sigma}_{25}^0 \\
&& \boldsymbol{\sigma}_{31}^0 \ \frac{\sigma_{32}^{a_{32}}}{a_{32}!} \ \frac{\sigma_{33}^{a_{33}}}{a_{33}!} \ \frac{\sigma_{34}^{1-a_{32}-a_{33}}}{(1-a_{32}-a_{33})!} \ \boldsymbol{\sigma}_{35}^0 \quad
\boldsymbol{\sigma}_{41}^0 \ \frac{\boldsymbol{\sigma}_{42}^{a_{42}}}{a_{42}!} \ \frac{\sigma_{43}^{a_{43}}}{a_{43}!} \ \frac{\sigma_{44}^{2-a_{42}-a_{43}}}{(2-a_{42}-a_{43})!} \ \boldsymbol{\sigma}_{45}^0 \\
&& \sigma_{51}^0 \ \frac{\boldsymbol{\sigma}_{52}^{2-a_{12}-a_{22}-a_{32}- a_{42}}}{(2-a_{12}-a_{22}-a_{32}- a_{42})!} \ \frac{\boldsymbol{\sigma}_{53}^{2-a_{13}-a_{23} -a_{33}-a_{43} }}{(2-a_{13}-a_{23} -a_{33}-a_{43} )!} \\
&& \frac{\boldsymbol{\sigma}_{54}^{-3 + a_{12} + a_{13} +a_{22} + a_{23} + a_{32} + a_{33} + a_{42} + a_{43}}}{(-3 + a_{12} + a_{13} +a_{22} + a_{23} + a_{32} + a_{33} + a_{42} + a_{43})!} \ \sigma_{55}^0
\end{eqnarray*}
where the null elementary moments are highlighted.

If a null elementary moment  $\sigma_{hk}$ has a null exponent, we set $\sigma_{hk}^0=1$. Otherwise, that is if $a_{hk} \neq 0$,  the corresponding addend is null. So that if one of the exponents of the  null elementary moments $\sigma_{hk}$, $(h,k) \neq (5,4)$ is non zero, then the addend is null.

Direct computation show that, if the exponents of the  null elementary moment  $\sigma_{hk}$, $(h,k) \neq (5,4)$ are all null, then $a_{54}=1$ so that also this addend is null.
\end{example}

\section{Final remarks}
\label{sec:conclusion}
We have presented an explicit method for computing the complex moments $\nu(\alpha)$,  $\alpha=(n_1,m_1,\dots,n_p,m_p)$,  of the multivariate complex normal. If $\alpha$ satisfies the hypotheses of Theorem~\ref{th:conditions}, the moment is zero. Otherwise, in Theorem~\ref{th:main} we proved a closed-form equation for $\nu(\alpha)$. This form is a polynomial in the individual covariances of total degree $\sum_{h=1}^p n_h$, easily computable and more practical than the recursive solution of the recurrent relations of Proposition~\ref{th:recurrence}.

The question of computing complex moments of the multivariate complex normal arised in the study of cubature formulae and numerical approximation of  integrals on the complex vector space. In~\cite{FRR:2017} novel cubature formulae with nodes on a suitable subset of the complex roots of the unity are discussed. The knowledge of the exact value of the  moments for the complex normal allows us for an evaluation of the approximation error and a comparison with respect to classical schemes.

\section*{Acknowlegments}
\label{sec:acknowlegments}

G. Pistone acknowledges the support of de Castro Statistics and Collegio Carlo Alberto. He is a member of GNAMPA-INDAM.
\bibliographystyle{apalike}
\bibliography{biblio_cub}

\begin{thebibliography}{}

\bibitem[Barvinok, 2007]{Barvinok07}
Barvinok, A. (2007).
\newblock Integration and optimization of multivariate polynomials by
  restriction onto a random subspace.
\newblock {\em Found. Comput. Math.}, 7(2):229--244.

\bibitem[{Fassino} et~al., 2017]{FRR:2017}
{Fassino}, C., {Riccomagno}, E., and {Rogantin}, M.-P. (2017).
\newblock {Cubature rules and expected value of some complex functions}.
\newblock {\em ArXiv e-prints}.

\bibitem[Goodman, 1963]{Goodman1963}
Goodman, N.~R. (1963).
\newblock Statistical analysis based on a certain multivariate complex
  {G}aussian distribution. ({A}n introduction).
\newblock {\em Ann. Math. Statist.}, 34:152--177.

\bibitem[It\^o, 1952]{ito:1952}
It\^o, K. (1952).
\newblock Complex multiple {W}iener integral.
\newblock {\em Jap. J. Math.}, 22:63--86 (1953).

\bibitem[Rudin, 1987]{rudin:1987RCA-III}
Rudin, W. (1987).
\newblock {\em Real and complex analysis}.
\newblock McGraw-Hill Book Co., New York, third edition.

\bibitem[Sultan and Tracy, 1996]{sultan|tracy:1996}
Sultan, S.~A. and Tracy, D.~S. (1996).
\newblock Moments of the complex multivariate normal distribution.
\newblock {\em Linear Algebra Appl.}, 237/238:191--204.
\newblock Special issue honoring Calyampudi Radhakrishna Rao.

\end{thebibliography}

\newpage
\titleformat{\section}{\large\bfseries}{\appendixname~\thesection .}{0.5em}{}
\begin{appendices}\label{App_L}


\section {Properties of the bounds $l_{ij}$ and $L_{ij}$.}

The following proposition  states that the definition of $I(\alpha)$ in Definition~\ref{def:Ialpha} is consistent.

\begin{proposition}\label{Prop:bound}
Let $\alpha$, $a$, $l_{ij}(\alpha,a)$ and $L_{ij}(\alpha,a)$ be  as in Equation \eqref{def:bounds}. It holds:
$$
0 \le l_{ij}(\alpha,a)  \le L_{ij }(\alpha,a)
$$
if the entries $a_{hk}$ of the vector $a$, for $h \le i, \; k \le j, \; (h,k) \neq (i,j)$, satisfy the corresponding bound $l_{hk}(\alpha,a) \le a_{hk} \le L_{hk}(\alpha,a)$.
\end{proposition}

\begin{proof} For simplicity $\alpha$ and $a$  are omitted. \\
Obviously, for each $i,j \le p$, we have  $l_{ij} \ge 0$ and so $a_{ij} \ge  0$.\\
If $(i,j) = (1,1)$,  $l_{11} =0 \wedge (n_1 - \sum_{k=2}^p m_k)$ and $L_{11}= n_1 \vee m_1$. If $L_{11}= n_1$, obviously $l_{11} \le L_{11}$.
The case $L_{11}=m_1$ also implies $l_{11} \le L_{11}$, since $\sum_{h=1}^p n_h =\sum_{k=1}^p m_k$ and so
 $n_1 \le \sum_{h=1}^p m_h$, that is  $n_1 - \sum_{h=2}^p m_h \le m_1$.  \\
The case with $(ij) \neq (1,1)$  is proved by induction.
\begin{itemize}
\item \emph{Base steps.}
\begin{itemize}
\item We prove that   $ l_{1,j} \le L_{1,j}$, with $j \le  p$,  by induction on $j$. The inequalities hold for $(i,j)=(1,1)$. We assume  $l_{1,(j-1)} \le a_{1,(j-1)} \le L_{1,(j-1)}$.\\
We have
$ a_{1,(j-1)} \le L_{1,(j-1)} \le  n_1 - \sum_{k=1} ^{j-2} a_{1k} $, that is $n_1 - \sum_{k=1} ^{j-1} a_{1k} \ge 0 $ and so $L_{1j}= (n_1 - \sum_{k=1} ^{j-1} a_{1k} ) \vee m_{j} \ge 0$.

We show that $L_{1j} \ge   n_1 -\sum_{k=j+1}^p m_k - \sum_{k=1} ^{j-1} a_{1k} $ and so, since $L_{1j} \ge 0$, we conclude that $L_{1j} \ge l_{1j}$.

If $L_{1j} = n_1 - \sum_{k=1} ^{j-1} a_{1k}$, obviously, $L_{1j} \ge  n_1 - \sum_{k=j+1}^p m_k - \sum_{k=1} ^{j-1} a_{1k} $.   The case $L_{1j}=m_j$ also implies $l_{1j} \le L_{1j}$. In fact, from the inductive hypothesis, we have $ a_{1,(j-1)} \ge l_{1,(j-1)} \ge  n_1 -\sum_{k=j}^p m_k - \sum_{k=1} ^{j-2} a_{1k} $, that is
$m_{j} \ge  n_1 - \sum_{k=j+1}^p m_k - \sum_{k=1} ^{j-1} a_{1k} $.

\item Analogously, the relation between  $l_{i,1}$ and $L_{i,1}$ can be shown.
\end{itemize}
\item \emph{Induction step.} \\
We show that $ l_{ij} \le L_{ij}$, by assuming that $ l_{hk} \le a_{hk} \le L_{hk}$  for $h < i$, $k\le p$, so that $ l_{(i-1),j} \le a_{(i-1),j} \le L_{(i-1),j}$,  and for $h \le p $, $k < j $, so that $ l_{i,(j-1)} \le a_{i,(j-1)} \le L_{i,(j-1)}$.
 It follows  that $L_{ij} \ge 0$ since the inequalities
\begin{eqnarray*}
 a_{(i-1),j} \le m_j -  \sum_{h=1} ^{i-2} a_{hj} \qquad \text{and} \qquad a_{i,(j-1)} \le n_i -  \sum_{k=1} ^{j-2} a_{ik}
  \end{eqnarray*}
imply
\begin{eqnarray*}
 m_j -  \sum_{h=1} ^{i-1} a_{hj} \ge 0 \qquad \text{and} \qquad n_i -  \sum_{k=1} ^{j-1} a_{ik} \ge 0 \ .
 \end{eqnarray*}
Furthermore, from $l_{i,(j-1)} \le a_{i,(j-1)}$ and from $l_{(i-1),j} \le a_{(i-1),j}$ it follows
\begin{eqnarray*}
 \sum_{h=1}^i n_h -  \sum_{k=j} ^{p}m_k -  \sum_{\substack{h=1\dots i \\ k=1\dots j-1 }} a_{hk} \le 0 \qquad \text{and} \qquad  \sum_{h=1}^{i-1}n_h -  \sum_{k=j+1} ^{p}m_k -  \sum_{\substack{h=1\dots i-1\\k=1\dots j}} a_{hk} \le 0
  \end{eqnarray*}
and so, by adding $m_j -\sum_{h=1}^{i-1} a_{hj} $ to both sides of the first relation and   $n_i -\sum_{k=1}^{j-1} a_{ik} $ to both sides of the second one
 \begin{eqnarray*}
 m_j -\sum_{h=1}^{i-1} a_{hj} &\ge & \sum_{h=1}^i n_h -  \sum_{k=j+1} ^{p}m_k -  \sum_{\substack{h=1\dots i ; k=1\dots j \\ (h,k) \neq (i,j)}} a_{hk}\\
n_i -\sum_{h=1}^{j-1} a_{ih} &\ge & \sum_{h=1}^i n_h -  \sum_{k=j+1} ^{p}m_k -  \sum_{\substack{h=1\dots i ; k=1\dots j \\ (h,k) \neq (i,j)}} a_{hk}
  \end{eqnarray*}
 We conclude that
 \begin{eqnarray*}
 L_{i,j}  =  \left( n_i -\sum_{k=1}^{j-1} a_{ik} \right ) \vee \left(m_j -\sum_{h=1}^{i-1} a_{hj}\right) \ge   \sum_{h=1}^i n_h -\sum_{h=j+1}^p m_h -  \sum_{\substack{h=1\dots i ; k=1\dots j-1\\ (h,k) \neq (i,j) }} a_{hk}
\end{eqnarray*}
and so, since $L_{ij} \ge 0$, $L_{ij} \ge l_{ij}$.
\end{itemize}
\end{proof}
Proposition~\ref{Prop:bound} also holds when some values $n_r$ and $m_s$ are equal to zero.
\begin{corollary}\label{cor:null_exp}
Let $\alpha=(n_1,m_1,\dots,n_p,m_p)$ be a multi-index. \\
If there exists an index $r$ such that
$n_r=0$ then $a_{rk}=0 $ for each $k=1\dots p$.\\
If there exists an index $s$ such that
$m_s=0$ then $a_{hs}=0 $ for each $h=1\dots p$.
\end{corollary}
\begin{proof}
From Proposition~\ref{Prop:bound}, $L_{r,p} \ge 0$ which implies
$$ 0 \le L_{r,p} \le n_r - \sum_{k=1}^{p-1} a_{rk} = - \sum_{k=1}^{p-1} a_{rk}$$
and so, since $a_{rk} \ge 0$ for each $ k\le p$, we conclude that $a_{rk} = 0$ for $r=1, \dots, (p-1)$. From this result it follows that $L_{r,p}=\left( n_r-\sum_{k=1}^{p-1} a_{rk} \right ) \vee \left(m_p -\sum_{h=1}^{r-1} a_{hp}\right)=0$, and so $a_{r,p}=0$.
Analogously we can show the thesis for $a_{hs}$, since $ 0 \le L_{p,s} \le m_s - \sum_{h=1}^{p-1} a_{hs}$.
\end{proof}

\begin{proposition}\label{extreme}
The elements of the vector $a$ with an index equal to $p$ have the following special form:
\begin{equation*}
  a_{ip} = n_i-\sum_{k=1}^{p-1}a_{ik}\ , \qquad
  a_{pj} = m_j-\sum_{h=1}^{p-1}a_{hj}\ , \qquad
  a_{pp} =  n_p-\sum_{k=1}^{p-1}  m_k + \sum_{h,k=1}^{p-1}a_{hk} \ .
\end{equation*}
\end{proposition}
\begin{proof} Proposition~\ref{Prop:bound} shows that $l_{ij}  \le L_{ij}$. We show, by induction, the thesis for $ a_{ip} $.\\
 Base step:  $ n_1-\sum_{k=1}^{p-1} a_{1k} =l_{1p}  \le L_{1p} =  m_1 \vee  n_1-\sum_{k=1}^{p-1} a_{1k} $ and so $ l_{1p}   = L_{1p}  = n_1 -\sum_{k=1}^{p-1} a_{1k} $ and $a_{1p}= n_1-\sum_{k=1}^{p-1} a_{1k} $. \\
 Induction step: let $  a_{hp} = n_h-\sum_{k=1}^{p-1}a_{hk} $ for $h < i$, that is $n_h = \sum_{k=1}^{p }a_{hk} $. We obtain, using the inductive hypothesis,
$$ l_{ip} = \sum_{h=1}^i n_h -   \sum_{\substack{h=1\dots i; k=1\dots p \\(h,k) \neq (i,p) }} a_{hk}  = \sum_{h=1}^{i-1} \left(n_h -   \sum_{k=1}^p   a_{hk} \right) +n_i -\sum_{ k=1}^{p-1} a_{ik} = n_i -\sum_{ k=1}^{p-1} a_{ik} $$
and so, since $l_{ip} \le L_{ip} =(n_i -\sum_{ k=1}^{p-1} a_{ik}) \vee (m_p-\sum_{ h=1}^{p-1} a_{hp})$, we conclude that $a_{ip}=n_i-\sum_{k=1}^{p-1}a_{ik}$. The proof for $a_{pj}$ is analogous. \\
Finally, from the special values of  $ a_{ip}$ and $ a_{pj} $, we obtain
\begin{eqnarray*}
l_{pp}&=&  \sum_{h=1}^p n_h - \sum_{\substack{h=1\dots p; k=1\dots p \\(h,k) \neq (p,p) }} a_{hk}  =
 \sum_{h=1}^p n_h - \sum_{h,k=1}^{p-1} a_{hk} - \sum_{h=1}^{p-1} a_{hp}  - \sum_{k=1}^{p-1} a_{pk}  \\
& =& \sum_{h=1}^p n_h - \sum_{h,k=1}^{p-1} a_{hk} - \sum_{h=1}^{p-1}  \left ( n_h-\sum_{k=1}^{p-1}a_{hk} \right) -  \sum_{k=1}^{p-1}\left(m_k-\sum_{h=1}^{p-1}a_{hk} \right)\\
&=&  n_p -  \sum_{k=1}^{p-1} m_k +  \sum_{k=1}^{p-1} \sum_{h=1}^{p-1}a_{hj} \ .
\end{eqnarray*}
Furthermore,
\begin{eqnarray*}
L_{pp}&=&\left (m_p - \sum_{h=1}^{p-1} a_{hp} \right) \vee \left (n_p - \sum_{k=1}^{p-1} a_{pk} \right) \\
&=& \left (m_p - \sum_{h=1}^{p-1} a_{hp} \right) \vee \left ( n_p -  \sum_{k=1}^{p-1} m_k +  \sum_{k=1}^{p-1} \sum_{h=1}^{p-1}a_{hj} \right)  \ .
\end{eqnarray*}
Analogously to the previous cases, the thesis follows since $l_{pp} \le L_{pp}$.
\end{proof}

\begin{proposition}\label{Prop:values}
Let $\alpha=(n_1,m_1,\dots, n_p,m_p) \in \mathbb Z_{\ge 0}^{2p}$ and let $a=(a_{1,1},\dots,a_{1,p},\dots,  a_{p,p}) \in \mathbb Z_{\ge 0}^{p^2}$ be a vector belonging to $I(\alpha)$.
\begin{enumerate}
\item Let $a_{ij}=L_{ij}(\alpha,a)$.
\begin{eqnarray*}
\text{If } \quad L_{ij}(\alpha,a) =  m_j - \sum_{h=1}^{i-1}a_{hj} \quad &\text{then}&\quad
a_{qj}= 0 \quad \forall \;  q=i+1,\dots, p \ .\\
\text{If }  \quad L_{ij}(\alpha,a)=  n_i  - \sum_{k=1}^{j-1}a_{ik} \quad &\text{then}&\quad
a_{it}=0 \quad \forall \;  t=j+1,\dots, p \ .
\end{eqnarray*}
\item Let $a_{ij}=l_{ij}(\alpha,a)$.
\begin{eqnarray*}
\text{If } \quad l_{ij}(\alpha,a)&=& \sum_{h=1}^ i n_h - \sum_{h=j+1}^ p m_h - \sum_{\substack{h=1,\dots, i \\ k=1, \dots, j \\ (h,k) \neq (i,j)}} a_{hk} \quad \text{then} \\
a_{it}&=& m_t-\sum_{h=1}^{i-1} a_{ht}\quad   \forall \;   {t=j+1,\dots,p } \ ,\\
a_{qj}&=& n_q-\sum_{k=1}^{j-1} a_{qk} \quad   \forall \;  {q=i+1,\dots, p}  \ ,\\
a_{qt}&=& 0 \qquad\qquad\qquad   \forall \;  {q=i+1,\dots, p}, \;{t=j+1,\dots,p } \ .
\end{eqnarray*}
\end{enumerate}
\end{proposition}

\begin{proof}
For simplicity $\alpha$ and $a$ are omitted.
\begin{enumerate}
\item We consider  $a_{ij}=L_{ij}=m_j-\sum_{h=1}^{i-1} a_{hj}$. Then
\begin{eqnarray} \label{hyp_2}
m_j=  \sum_{h=1}^i a_{hj} \ .
\end{eqnarray}
We show, by induction, that  $a_{qj} = 0$ for $q >i$.
\begin{itemize}
\item \emph{Base step: $q=i+1$}. We have $0 \le a_{(i+1),j} \le L_{(i+1),j} \le  m_j-\sum_{h=1}^i  a_{hj} = 0$.
\item \emph{Induction step.} Let $a_{hj}=0$ for $h=(i+1),\dots (q-1)$. Then
$$ 0 \le a_{qj}  \le L_{qj} \le m_j - \sum_{ h=1}^{q-1} a_{hj}  = m_j - \sum_{ h=1}^{i} a_{hj}  =  0 \ . $$
\end{itemize}
Analogously, when $a_{ij}=L_{ij}=n_i - \sum_{h=1}^{j-1}a_{ih}$, then $a_{it}=0$ for all $t=j+1,\dots, p$.
\item  First of all, we show by induction that  $a_{it}=m_t-\sum_{h=1}^{i-1} a_{ht} $ for each $ t> j$. \\
Since $a_{ij}=   \sum_{h=1}^ i n_h - \sum_{k=j+1}^ p m_k - \sum_{\substack{h=1,\dots, i; k=1, \dots, j \\ (h,k) \neq (i,j)}} a_{hk} $ then
\begin{equation*}
  \sum_{h=1}^ i n_h - \sum_{k=j+1}^ p m_k  = \sum_{\substack{h=1,\dots, i;\\  k=1, \dots, j} } a_{hk} \
\end{equation*}
and so, for each $t>j$ it holds
\begin{multline} \label{hyp}
 l_{i,t} = \sum_{h=1}^i n_h - \sum_{k=t+1}^p m_k - \sum_{h=1}^i  \sum_{\substack{k=1\\ (h,k) \neq (i,t)}} ^ta_{hk} = \sum_{h=1}^i n_h - \sum_{k=j+1}^p m_k + \sum_{k=j+1}^t m_k -  \sum_{h=1}^i \sum_{\substack{k=1\\ (h,k) \neq (i,t)}}^t a_{hk} \\
=  \sum_{k=j+1}^t m_k  + \sum_{h=1}^i \sum_{k=1}^j    a_{hk}-  \sum_{h=1}^i \sum_{\substack{k=1\\ (h,k) \neq (i,t)}}^t a_{hk} =  \sum_{k=j+1}^t m_k  - \sum_{h=1}^{i-1} \sum_{k=j+1}^t  a_{hk}-  \sum_{k=j+1}^{t-1} a_{ik} \\
= \sum_{k=j+1}^t (m_k- \sum_{h=1}^{i-1}  a_{hk }) -  \sum_{k=j+1}^{t-1} a_{ik}
\end{multline}
\begin{itemize}
\item \emph{Base step: $t=j+1$.} \\
By Equation~\eqref{hyp} with $t=j+1$ we have
$$  l_{i,j+1} = m_{j+1} -   \sum_{h=1}^{ i-1 } a_{h,j+1}  \le L_{i,j+1} \le  m_{j+1} -   \sum_{h=1}^{ i-1 } a_{h,j+1}  $$
that is $l_{i,j+1} = L_{i,j+1}$, and so the thesis follows since  $l_{i,j+1} \le a_{i,j+1} \le  L_{i,j+1}$.
\item \emph{Induction step.} \\
Let  $a_{ik}=m_k - \sum_{h=1}^{i-1} a_{hk}$ for  $j+1 \le k< t$.
By Equation~\eqref{hyp} and by the inductive hypothesis,
\begin{eqnarray*}
 l_{it} &= & m_t- \sum_{h=1}^{i-1}  a_{ht } +  \sum_{k=j+1}^{t-1} (m_k - \sum_{h=1}^{i-1}  a_{hk }) -  \sum_{k=j+1}^{t-1} a_{ik} \\
 &=& m_t- \sum_{h=1}^{i-1}  a_{ht } +  \sum_{k=j+1}^{t-1} a_{ik}  -  \sum_{k=j+1}^{t-1} a_{ik}  =  m_t- \sum_{h=1}^{i-1}  a_{ht } \ge L_{it} \ge l_{it}
 \end{eqnarray*}
 and so  $l_{it} = L_{it}$, and the thesis follows since  $l_{it} \le a_{it} \le  L_{it}$.
\end{itemize}
Analogously, we can show, by induction, that  $a_{qj}= n_q-\sum_{k=1}^{j-1} a_{qk}$ for $q > i $.

Furthermore, since $a_{qj}=L_{qj} = n_q-\sum_{k=1}^{j-1} a_{qk}$ for $q=i+1,\dots, p $, then from Item 1 it follows  $a_{qt}=0$ for  $t= j+1,\dots, p$ and so, by varying $q  \ge i + 1$ we obtain the thesis.
\end{enumerate}
\end{proof}
\begin{proposition}\label{Prop:bound_b}
Let $r$, $s$ be two integers in $\{ 1,\dots,p\}$, let $\alpha_{rs}^-=\alpha-e_{2r-1}-e_{2s}$, as  defined before,  and let  $I(\alpha)$ be as in Definition~\eqref{def:Ialpha}.
 Let $I^+_{rs}= \left \{ a_{rs}^+ \vert a \in I( \alpha_{rs}^-) \right \}$. It holds
 \begin{eqnarray*}
I^+_{rs} &=&\{ b \in  I(\alpha) \ | \ b_{ij}  \neq \sum_{h=1}^i n_h -\sum_{k=j+1} ^p m_k -\sum_{\substack{h=1,\dots,i; j=1,\dots,j\\ (h,k) \neq (i,j)}} b_{hk};\\
 b_{is} &\neq& m_s  - \sum_{h=1}^{i-1} a_{hs}; \qquad b_{rj} \neq  n_r  - \sum_{k=1}^{j-1} b_{rk}; \qquad b_{rs} \neq 0 \qquad   i< r, j<s   \}  \ .
  \end{eqnarray*}
\end{proposition}
\begin{proof}
Let $b=a_{rs}^+$, that is $b_{ij}=a_{ij}$, if $(i,j) \neq (r,s)$, and $b_{rs}=a_{rs}+1$.
 \begin{enumerate}
\item Let  $i < r $ and $j < s$. Both $l_{ij}(\alpha_{rs}^-,a)$ and $L_{ij}(\alpha_{rs}^-,a)$ do not involve $a_{rs}$, so that $l_{ij}(\alpha_{rs}^-,a)=l_{ij}(\alpha_{rs}^-,b)$ and $L_{ij}(\alpha_{rs}^-,a)=L_{ij}(\alpha_{rs}^-,b)$.\\
The bound $l_{ij}(\alpha_{rs}^-,a)$ only involves $m_s$  and, since  $L_{ij}(\alpha_{rs}^-,a) $ does not involve $n_r$, $m_s$,   $L_{ij}(\alpha_{rs}^-,a) = L_{ij}(\alpha,a) $ and so
\begin{eqnarray*}
l_{ij}(\alpha_{rs}^-,a) =0 \wedge \left(\sum_{h=1}^i n_h - \sum_{k=j+1}^p m_k  +1 -\sum_{\substack{h=1\dots,i;k=1,\dots,j \\ (h,k) \neq (i,j) }}a_{hk} \right) \le a_{ij} \le L_{ij}(\alpha_{rs}^-,a) =L_{ij}(\alpha,a) 
\end{eqnarray*}
or, equivalently,
\begin{eqnarray*}
0 \wedge \left(\sum_{h=1}^i n_h - \sum_{k=j+1}^p m_k  +1 -\sum_{\substack{h=1\dots,i;k=1,\dots,j \\ (h,k) \neq (i,j) }}b_{hk} \right) \le b_{ij} \le L_{ij}(\alpha,b) \ .
\end{eqnarray*} %
\item Let  $i < r $ and $j = s$. Both $l_{is}(\alpha_{rs}^-,a)$ and $L_{is}(\alpha_{rs}^-,a)$ do not involve $a_{rs}$, so that $l_{is}(\alpha_{rs}^-,a)=l_{is}(\alpha_{rs}^-,b)$ and $L_{is}(\alpha_{rs}^-,a)=L_{is}(\alpha_{rs}^-,b)$.\\
Since $l_{is}(\alpha_{rs}^-,a)$ does not involve $n_r$, $m_s$  then $l_{is}(\alpha_{rs}^-,a)=l_{is}(\alpha,a)$;
furthermore, $L_{is}(\alpha_{rs}^-,a)$ does not involve $n_r$ , and so
$$ l_{is}(\alpha,a)= l_{is}(\alpha_{rs}^-,a) \le a_{is} \le L_{is}(\alpha_{rs}^-,a)= \left(n_i - \sum_{k=1}^{s-1} a_{ik} \right) \vee \left( m_s -1- \sum_{h=1}^{i-1} a_{hs}\right) \ ,$$
or, equivalently,
$$ l_{is}(\alpha,b) \le b_{is} \le\left(n_i - \sum_{k=1}^{s-1} b_{ik} \right) \vee \left( m_s -1- \sum_{h=1}^{i-1} b_{hs}\right) \ .$$
\item Let  $i = r $ and $j < s$. Analogously  to Item~2, we have
$$ l_{rj}(\alpha,b)=  \le b_{rj} \le  \left(n_r-1 - \sum_{k=1}^{j-1} a_{rk} \right) \vee \left( m_j - \sum_{h=1}^{r-1} a_{hj}\right)  \ .
$$
\item Let  $i = r $ and $j = s$. Both $l_{rs}(\alpha_{rs}^-,a)$ and $L_{rs}(\alpha_{rs}^-,a)$ do not involve $a_{rs}$, so that $l_{rs}(\alpha_{rs}^-,a)=l_{rs}(\alpha_{rs}^-,b)$ and $L_{rs}(\alpha_{rs}^-,a)=L_{rs}(\alpha_{rs}^-,b)$.\\
Since $l_{rs}(\alpha_{rs}^-,a)$  only involves $n_r$, while $L_{rs}(\alpha_{rs}^-,a)$  involves $n_r$ and $m_s$, we have
\begin{eqnarray*}
&&l_{rs}(\alpha_{rs}^-,a) =0 \wedge \left (\sum_{h=1}^r n_h -1- \sum_{k=s+1}^p m_k -\sum_{\substack{h=1\dots,r;k=1,\dots,s \\ (h,k) \neq (r,s) }}a_{hk} \right ) \le a_{rs}  \\
&&\le L_{rs}(\alpha_{rs}^-,a) = \left (n_r-1- \sum_{k=1}^{s-1}a_{rk}\right ) \vee \left(m_s-1- \sum_{h=1}^{r-1}a_{hs}\right)=L_{ij}(\alpha,a) -1  \ ,
\end{eqnarray*}
or, equivalently, since $b_{rs} = a_{rs}+1$,
\begin{eqnarray*}
&&1 \wedge \left (\sum_{h=1}^r n_h - \sum_{k=s+1}^p m_k -\sum_{\substack{h=1\dots,r;k=1,\dots,s \\ (h,k) \neq (r,s) }}b_{hk} \right ) \le b_{rs}  \\
&&\le \left (n_r- \sum_{k=1}^{s-1}b_{rk}\right ) \vee \left(m_s- \sum_{h=1}^{r-1}b_{hs}\right) = L_{rs}(\alpha,b) \ .
\end{eqnarray*}
\item We consider the case with $i>r$ or $j>s$.
\begin{description}
\item{-} Let  $i > r$. Then $l_{ij}(\alpha_{rs}^-,a)$  involves $n_r$, and $n_i -\sum_{h=1}^{j-1} a_{ih}$ does not involve $n_r$ nor~$a_{rs}$.\\
 If $j<s$ then $l_{ij}(\alpha_{rs}^-,a)$ involves $m_s$ but it does not involve $a_{rs}$ and so, by direct computation,
 $l_{ij}(\alpha_{rs}^-,a)=l_{ij}(\alpha,a)=l_{ij}(\alpha,b)$. Furthermore, $m_j -\sum_{h=1}^{i-1} a_{hj}$ does not involves $m_s$ and $a_{rs}$, and so $L_{ij}(\alpha_{rs}^-,a)=L_{ij}(\alpha,a)=L_{ij}(\alpha,b)$. \\
If $j \ge s$ then $l_{ij}(\alpha_{rs}^-,a)$ involves $a_{rs}$ but it does not involve $m_s$ and so, by direct computation, $l_{ij}(\alpha_{rs}^-,a)=l_{ij}(\alpha,b)$. Furthermore, if $j=s$ then $L_{is}(\alpha_{rs}^-,a)$ involves $m_s$ and  $a_{rs}$, and so $L_{is}(\alpha_{rs}^-,a)=L_{is}(\alpha,b)$; otherwise, if $j<s$  it does not involve $m_s$ nor $a_{rs}$ and so
$L_{ij}(\alpha_{rs}^-,a)=L_{ij}(\alpha,a)=L_{ij}(\alpha,b)$.
\item{-} Let   $j > s$. The proof is analogous to the case $i > r$.
\end{description}
Then we obtain
$$ l_{ij}(\alpha,b) \le b_{ij} \le L_{ij}(\alpha,b) \ . $$
\end{enumerate}
The thesis follows from previous bounds for the elements of the vector $b$, since they coincide with the bounds of the elements of the vector $a$, except for some values at the ends of the intervals.
\end{proof}

\end{appendices}
\end{document}